\documentclass[a4paper,12pt]{article}
\usepackage{amsfonts}
\usepackage[T1]{fontenc}
\usepackage{microtype}
\usepackage{makeidx}
\usepackage{setspace}
\usepackage{subcaption}
\usepackage{authblk}
\usepackage{graphicx}
\usepackage{epstopdf}
\usepackage[all]{xy}
\usepackage{amsmath}
\usepackage{amssymb}
\usepackage{booktabs}
\usepackage{braket}
\usepackage{array}
\usepackage{tabularx}
\usepackage{multirow}
\usepackage{indentfirst}
\usepackage{cite}
\usepackage{stmaryrd}
\usepackage{faktor}
\usepackage{titling}
\usepackage{tikz}
\usepackage{dsfont}
\usepackage{hyperref}
\usepackage{url}
\usetikzlibrary{matrix,arrows,decorations.pathmorphing}
    \renewcommand{\leq}{\leqslant}
    
\usepackage{amsthm}
\theoremstyle{plain}
\newtheorem{thm}{Theorem}[section]
\newtheorem{dfn}[thm]{Definition}

\newtheorem{prop}[thm]{Proposition}
\newtheorem{cor}[thm]{Corollary}
\newtheorem{ex}[thm]{Example}
\newtheorem{conge}[thm]{Conjecture}
\newtheorem{prob}[thm]{Problem}
\newtheorem*{thmintro}{Theorem}

\theoremstyle{remark}
\newtheorem{oss}[thm]{Remark}

\DeclareMathOperator{\End}{End}

\DeclareMathOperator{\N}{\mathbb{N}}

\def\<{\leqslant_{L^k}}
\def\X{\leqslant_{L^X}}

\DeclareMathOperator{\var}{\vartriangleleft}

\title{ The intermediate orders of a Coxeter group}
\author{}
\author{Angela Carnevale\thanks{School of Mathematical and Statistical Sciences, National University of  Ireland, Galway. \href{mailto:angela.carnevale@nuigalway.ie}{angela.carnevale@nuigalway.ie} }, Matthew Dyer\thanks{Department of Mathematics, University of Notre Dame. \href{mailto:dyer@nd.edu}{dyer@nd.edu}} \ and Paolo Sentinelli\thanks{ Dipartimento di Matematica, Politecnico di Milano, Milan, Italy. \\ \href{mailto:paolosentinelli@gmail.com}{paolosentinelli@gmail.com}}}
\date{}
\begin{document}
\maketitle

\vspace{-4em}

\begin{abstract}
  We define a class of partial orders on a Coxeter group associated
  with sets of reflections. In special cases, these lie between the
  left weak order and the Bruhat order. We prove that these posets are
  graded by the length function and that the projections on the right
  parabolic quotients are always order preserving. We also introduce
  the notion of $k$-Bruhat graph, $k$-absolute length and $k$-absolute
  order, proposing some related conjectures and problems.
\end{abstract}

\section{Introduction}\label{intro}

The weak order and the Bruhat order of a Coxeter group are partial
orders of preeminent importance in wide parts of algebraic
combinatorics, representation theory and algebraic geometry. These
orders depend on the Coxeter presentation of the group and the Bruhat
order is a refinement of the weak order, once a presentation is
chosen.  Both orders are graded by the length function of the
group. The weak order is a complete meet-semilattice (and an
orthocomplemented lattice in the finite case) and the order complex of
its open intervals are homotopy equivalent to spheres or are
contractible. On the other hand, the Bruhat order is Eulerian and its
open intervals are shellable.  See \cite[Ch.~2 and 3]{BB} and
references therein for these and other properties.

In this article we introduce a new class of partial orders on any
Coxeter group associated with sets of reflections. For the sake of
simplicity, we illustrate here a special case of particular
interest. For each $k\in \N$ we define an order $\<$ on a Coxeter
group $W$ associated with the set of reflections whose length is
bounded (in a way that depends on $k$). If $a < b$ then
$(W,\leqslant_{L^b})$ is a refinement of $(W,\leqslant_{L^a})$, for
all $a,b \in \N$. Moreover, $\leqslant_{L^0}=\leqslant_L$ is the left
weak order on $W$. If $\leqslant$ is the Bruhat order and $k\in \N$,
we obtain a sequence of injective poset morphisms
$$(W,\leqslant_L) \hookrightarrow (W,\leqslant_{L^1}) \hookrightarrow \ldots  \hookrightarrow (W,\leqslant_{L^k}) \hookrightarrow (W,\leqslant).$$
For this reason we call the new orders in this special case \emph{$k$-intermediate orders}; cf.\ Definition \ref{def ordini}.

Our first main result pertaining to these orders is the following. We
will prove it as a consequence of a more general result; see Theorem
\ref{corollario graduato} and Corollary~\ref{cor:intermediate}.
\begin{thmintro}
  The poset $(W,\<)$ is graded by the Coxeter length for all $k\in
  \N$. \end{thmintro}  Consider the function $P^J: W \rightarrow W$
which assigns to an element $w\in W$ the representative of minimal
length of the coset $wW_J$, where $W_J\subseteq W$ is the parabolic
subgroup generated by $J$. Our second main result for $k$-intermediate orders is the following (cf.\ Theorem~\ref{ordine preservato}).
\begin{thmintro}
  The functions $P^J$ are order preserving on $(W,\<)$ for all $k\in\mathbb N$.
  \end{thmintro}
  After setting up some notation and recalling some preliminaries in
  Section~\ref{sec:preliminaries}, we will prove these theorems as
  consequences of more general results pertaining to a broader family
  of posets; see Section~\ref{sec:intermediate}.

  Our $k$-intermediate orders are defined by considering sets of
  reflections with bounded length, as the Bruhat order is defined by
  considering the whole set of reflections. In the same spirit, we
  define in Section~\ref{sec:pbs} various other $k$-analogues of
  related objects.  In particular, we define the $k$-\emph{absolute
    orders} (see Definition \ref{def k-assoluto}). Both the
  $k$-intermediate and $k$-absolute orders coincide with the weak
  order if $k=0$. The absolute order of the symmetric group was
  introduced by T.~Brady in \cite{brady} and it is involved in the
  construction of Eilenberg-MacLane spaces for the braid groups. He
  also proved that the lattice of noncrossing partitions $NC_n$ is
  isomorphic to any of the maximal intervals in the absolute order of
  $S_n$ corresponding to Coxeter elements. The absolute order on a
  Coxeter group has been subsequently considered in several kinds of
  problems concerning shellability, Cohen-Macaulay, Sperner and
  spectral properties, among others; see e.g. \cite{Athanasiadis},
  \cite{Athanasiadis2}, \cite{sperner2}, \cite{harper},
  \cite{spectra}. Inspired by these works we formulate some problems
  and conjectures about $k$-intermediate orders, $k$-absolute orders
  and their rank function, which we call $k$-\emph{absolute length}.

  We also introduce the $k$-Bruhat graph of a Coxeter system, which
  turns out to be always locally finite and in some sense approximates
  the Bruhat graph when the group is infinite. Recent results on
  the Ricci curvature of Bruhat and Cayley graphs of Coxeter groups
  (\cite{Siconolfi-1}, \cite{Siconolfi-2}) lead to consider the Ricci
  curvature of a $k$-Bruhat graph; this problem closes the paper.

\section{Notation and preliminaries}\label{sec:preliminaries}

In this section we establish some notation and we collect some basic
results from the theory of Coxeter systems which are useful in the
sequel. The reader can consult \cite{BB} and references therein for
further details. We follow \cite[Ch.~3]{StaEC1} for notation and
terminology concerning posets.

With $\mathbb{N}$ we denote the set of non-negative integers. For
$n\in \mathbb{N}$ we let $[n]:=\Set{1,2,\dots,n}$. With $\biguplus$ we
denote the disjoint union and with $|X|$ the cardinality of a set
$X$. Given any category, $\End(O)$ denotes the set of endomorphisms of
an object $O$.

Let $(W,S)$ be a Coxeter system. That is, $W$ is a group with a
presentation given by a finite set of involutive generators $S$ and
relations encoded by a \emph{Coxeter matrix}
$m:S\times S \rightarrow \{1,2,\dots,\infty\}$ (see
\cite[Ch.~1]{BB}). A Coxeter matrix over $S$ is a symmetric matrix
which satisfies the following conditions for all $s,t\in S$:
\begin{enumerate}
  \item $m(s,t)=1$ if and only if $s=t$;
  \item $m(s,t)\in \{2,3,\dots,\infty\}$ if $s\neq t$.
\end{enumerate}  The presentation
$(W,S)$ of the group $W$ is then the following:
$$\left\{
  \begin{array}{ll}
    \mathrm{generators}: & \hbox{$S$;} \\
    \mathrm{relations}: & \hbox{$(st)^{m(s,t)}=e$,}
  \end{array}
\right.$$ for all $s,t\in S$, where $e$ denotes the identity in
$W$. The Coxeter matrix $m$ attains the value $\infty$ at $(s,t)$ to
indicate that there is no relation between the generators $s$ and $t$.
The class of words expressing an element of $W$ contains words of
minimal length; the \emph{length function}
$\ell: W \rightarrow \mathbb{N}$ assigns to an element $w\in W$ such
minimal length. The identity $e$ is represented by the empty word and
then $\ell(e)=0$. A \emph{reduced word} or \emph{reduced expression}
for an element $w\in W$ is a word of minimal length representing $w$.
The set of \emph{reflections} of $(W,S)$ is defined by
$T:=\{wsw^{-1}:w\in W, \, s\in S\}$. If $J\subseteq S$ and $v\in W$,
we let
\begin{gather*} W^J:=\Set{w\in W:\ell(w)<\ell(ws)~\forall~s\in J},
\\ {^JW}:=\Set{w\in W:\ell(w)<\ell(sw)~\forall~s\in J},
\\ D_L(v):=\Set{s\in S:\ell(sv)<\ell(v)},
\\ D_R(v):=\Set{s\in S:\ell(vs)<\ell(v)}.
\end{gather*}
With $W_J$ we denote the subgroup of $W$ generated by $J\subseteq S$;
such a group is usually called a \emph{parabolic subgroup} of $W$. In
particular, $W_S=W$ and $W_\varnothing = \Set{e}$.

Given a Coxeter system $(W,S)$, we let $\leqslant_L$ and $\leqslant$
be the \emph{left weak order} and the \emph{Bruhat order} on $W$,
respectively. The covering relations of the left weak order are
characterized as follows: $u \var v$ if and only if $\ell(u)<\ell(v)$
and $uv^{-1} \in S$.  The covering relations of the Bruhat order are
characterized as follows: $u \var v$ if and only if
$\ell(u)=\ell(v)-1$ and $uv^{-1} \in T$.  The posets $(W,\leqslant_L)$
and $(W,\leqslant)$ are graded with rank function $\ell$ and
$(W,\leqslant_L) \hookrightarrow (W,\leqslant)$.  We recall a
characterizing property of the Bruhat order, known as \emph{lifting
  property} (see \cite[Proposition~2.2.7]{BB}):
\begin{prop}[Lifting Property] \label{sollevamento} Let $v,w\in W$
  such that $v<w$ and $s\in D_L(w)\setminus D_L(v)$. Then
  $v\leqslant sw$ and $sv\leqslant w$.
\end{prop} For $J\subseteq S$, each element $w\in W$ factorizes
uniquely as $w=w^Jw_J$, where $w^J\in W^J$, $w_J\in W_J$ and
$\ell(w)=\ell(w_J)+\ell(w^J)$; see~\cite[Proposition~2.4.4]{BB}. We
consider the idempotent function $P^J:W \rightarrow W$ defined by
\begin{equation*} P^J(w)=w^J,
\end{equation*} for all $w\in W$. This function is order preserving for the Bruhat order
(see \cite[Proposition~2.5.1]{BB}). In the next section
we prove that the function $P^J$ is order preserving for a wider class of partial orders which we are going to introduce. 
In a similar way, one defines an order-preserving function $Q^J:(W,\leqslant) \rightarrow
(W,\leqslant)$ by setting $Q^J(w)={^J w}$, where $w=w'_J {^J w}$ with
$w'_J\in W_J$, $^Jw\in {^JW}$ and $\ell(w)=\ell(w'_J)+\ell(^Jw)$.

We end this section by recalling a theorem about the Bruhat graph of a
Coxeter system $(W,S)$ and the reflection subgroups of $W$. The
\emph{Bruhat graph} of $\left(W,S\right)$ is the directed graph
$\Omega_{W,S}$ whose vertex set is $W$ and such that there is an arrow
from $u$ to $v$ if and only if $uv^{-1} \in T$ and $u<v$; such an
arrow is labeled by the reflection $uv^{-1}$. If $Y \subseteq W$, we
denote by $\Omega_{W,S}(Y)$ the induced subgraph of $\Omega_{W,S}$
whose vertex set is $Y$. A subgroup $W' \subseteq W$ generated by
$X \subseteq T$ is called a \emph{reflection subgroup} of $W'$ and
$\left(W',S' \right)$ is a Coxeter system, where
$S':= \left\{t \in T: N(t) \cap W' =\{t\}\right\}$ and
$N(v):=\Set{t\in T:\ell(tv)<\ell(v)}$; see \cite{Dyer}. The following
theorem is \cite[Theorem 1.4]{Dyer2}.
\begin{thm} \label{teorema Dyer} Let $\left(W,S\right)$ be a Coxeter
  system and $W'\subseteq W$ a reflection subgroup with set of Coxeter
  generators $S'$. Then
  \begin{enumerate}
      \item $\Omega_{W',S'}=\Omega_{W,S}(W')$;
      \item for all $x\in W$ there exists $x_0 \in W'x$ such that the
        function $W' \rightarrow W'x$ given by $w \mapsto wx_0$
        induces an isomorphism of directed graphs between
        $\Omega_{W',S'}$ and $\Omega_{W,S}(W'x)$ which preserves the
        labels of the edges.
  \end{enumerate}
\end{thm}

\section{Intermediate orders}\label{sec:intermediate}

In this section we introduce the main objects of our study. 
Let $(W,S)$ be a Coxeter system and $T$ its set of reflections.
For $k\in \mathbb{N}$ we let $$T_k :=\left\{t\in T: \frac{\ell(t)-1}{2}\leqslant k\right\}.$$

The following definition introduces the notion of \emph{$k$-intermediate order}. 
\begin{dfn} \label{def ordini} Let $X\subseteq T$.
  We define a partial order $\X$ on $W$ by letting $u\X v$ if and only if
  \begin{itemize}
    \item $u=v$ or
    \item $\ell(u)<\ell(v)$ and there exist $t_1,\ldots,t_r \in X$ such that $$u<t_1u < t_2t_1u<\dots< t_r \cdots t_1u=v.$$  
    \end{itemize} For $k\in \N$ and $X=T_k$, we denote $\X$ by $\<$
    and call it a \emph{$k$-intermediate order} on $W$.
\end{dfn}

Note that $(W,\leqslant_{L^0})=(W,\leqslant_L)$ and, if $W$ is finite,
for $k$ big enough $(W,\<) = (W,\leqslant)$.  We have that
$(W,\leqslant_{L^0}) \hookrightarrow (W,\leqslant_{L^1})
\hookrightarrow \ldots \hookrightarrow (W,\leqslant)$, which justifies
the name `$k$-intermediate orders'. For $u,v \in W$ such that $u\<v$,
we denote by $[u,v]_k$ the corresponding interval in $(W,\<)$ and by
$[u,v]$ the interval in $(W,\leqslant)$.
 
 To ease the notation, we write $\vartriangleleft_X $ for a covering relation in $(W,\leqslant_{L^X})$.

\begin{figure}[htb]\centering\begin{tikzpicture}[scale=.6]

\matrix (a) [matrix of math nodes, column sep=0cm, row sep=1 cm]{
       &       &       &       &       &  4321 &     &     &     &       &     \\
       &       &       & 4312  &       &  4231 &     & 3421&     &       &     \\
       &  4132 &       & 4213  &       &  3412 &     & 2431&     & 3241  &     \\
1432   &       &  4123 &       & 2413  &       & 3142&     & 3214&       & 2341\\
       &  1423 &       & 1342  &       &  2143 &     & 3124&     & 2314  &     \\
       &       &       & 1243  &       &  1324 &     & 2134&     &       &     \\
       &       &       &       &       &  1234 &     &     &     &       &     \\};

\foreach \i/\j in {1-6/2-4, 1-6/2-6, 1-6/2-8,%
2-4/3-2,2-4/3-6,2-6/3-4, 2-6/3-8, 2-8/3-6, 2-8/3-10,%
3-2/4-1,3-2/4-3,3-4/4-3,3-4/4-5,3-6/4-7,3-8/4-5,3-8/4-11,3-10/4-9,3-10/4-11,%
4-1/5-2,4-1/5-4,4-3/5-2,4-5/5-6,4-7/5-4,4-7/5-8,4-9/5-8,4-9/5-10,4-11/5-10,%
5-2/6-4,5-4/6-6,5-6/6-4,5-6/6-8, 5-8/6-6, 5-10/6-8,%
6-4/7-6,6-6/7-6,6-8/7-6,%
5-2/4-5,5-2/6-6, 6-4/5-4, 6-6/5-10, 6-8/5-8, 5-6/4-3, 5-6/4-11, 5-10/4-5,
4-1/3-6, 4-7/3-2, 4-7/3-10, 4-9/3-6,%
3-2/2-6, 3-4/2-4, 3-8/2-8, 3-10/2-6}
    \draw (a-\i) -- (a-\j);

\end{tikzpicture} \caption{Hasse diagram of $(S_4,\leqslant_{L^1})$.} \label{fig-S4} \end{figure}
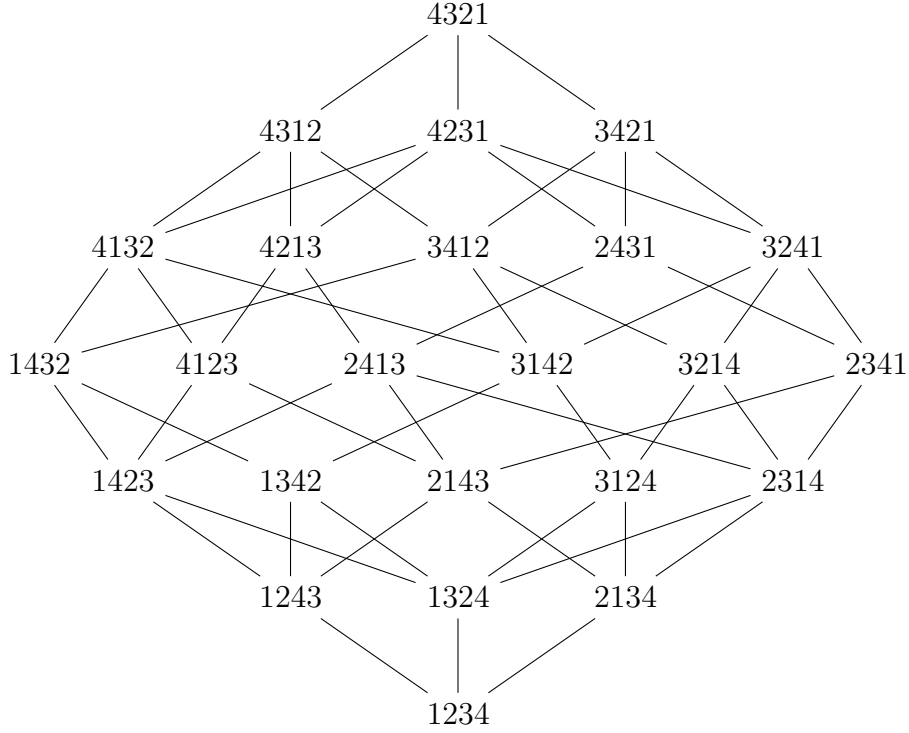
\begin{ex}
  Let $(W,S)=(S_4,\{s_1,s_2,s_3\})$, the symmetric group of order $24$
  with its standard Coxeter presentation. Then
  $(S_4,\leqslant_{L^2}) =(S_4,\leqslant)$, whose Hasse diagram is
  displayed in \cite[Figure 2.4]{BB}. The Hasse diagram of the poset
  $(S_4,\leqslant_{L^0})=(S_4,\leqslant_L) $, i.e. the Cayley graph of
  $(S_4,S)$, appears in \cite[Figure 3.2]{BB}.  The poset
  $(S_4,\leqslant_{L^1})$ is depicted in Figure \ref{fig-S4}.
\end{ex}

In order to prove the first main result of this paper, we define a
poset $\left(T,\sqsubseteq\right)$ as follows. Let $\Omega_{W,S}$ be
the Bruhat graph of $\left(W,S\right)$; we define a relation
$\sqsubseteq'$ on $T$ by setting $t \sqsubseteq' t'$ if and only if
$t=t'$ or there exists a reflection subgroup $W'\subseteq W$ such
that:
\begin{enumerate}
    \item $t,t'\in W'$;
    \item the system $(W',S')$ is dihedral, i.e. $|S'|=2$;
    \item the distance $\vec{d}(t,t')$ in $\Omega_{W,S}(W')=\Omega_{W',S'}$ is finite.
    \end{enumerate} We define $\sqsubseteq$ as the transitive closure
    of the relation $\sqsubseteq'$.  By definition,
    $\left(T,\sqsubseteq\right) \hookrightarrow
    \left(T,\leqslant\right)$ as posets, where $\leqslant$ is the
    Bruhat order. The Hasse diagrams of $(T,\sqsubseteq)$ in types $A_3$ and $B_3$ are displayed in Figure~\ref{fig:A3B3}.

\begin{figure}[htb]%
\begin{subfigure}[b]{.5\textwidth}\centering
 \begin{tikzpicture}[scale=1]
  \node (a1) at (2,2) {$s_1s_2s_3s_2s_1$};
  \node (b1) at (1,1) {$s_1s_2s_1$};
  \node (b2) at (3,1) {$s_2s_3s_2$};
  \node (c1) at (0,0) {$s_1$};
  \node (c2) at (2,0) {$s_2$};
  \node (c3) at (4,0) {$s_3$};
  \draw (a1) -- (b1) (a1) -- (b2) (b1) -- (c1) (b1) -- (c2) (b2)--(c2) (b2)--(c3);
\end{tikzpicture} \caption{ $(T,\sqsubseteq)$ in type $A_3$.}
 \end{subfigure}\begin{subfigure}[b]{.5\textwidth}
 \begin{tikzpicture}[scale=1]
  \node (a1) at (1,3) {$s_1s_0s_1s_2s_1s_0s_1$};
  \node (b1) at (0,2) {$s_0s_1s_2s_1s_0$};
  \node (b2) at (3,2) {$s_2s_1s_0s_1s_2$};
  \node (c1) at (0,1) {$s_0s_1s_0$};
  \node (c2) at (2,1) {$s_1s_0s_1$};
  \node (c3) at (4,1) {$s_1s_2s_1$};
  \node (d1) at (0,0) {$s_0$};
  \node (d2) at (2,0) {$s_1$};
  \node (d3) at (4,0) {$s_2$};
  \draw (a1) -- (b1) (b1) -- (c1) (a1) -- (c2) (b1) -- (c3) (b2)--(c2) (b2)--(c3) (c1) -- (d1) (c1) -- (d2) (c2) -- (d1) (c2) -- (d2)
  (c3) -- (d2) (c3) -- (d3);
\end{tikzpicture}\caption{$(T,\sqsubseteq)$ in type $B_3$.}\label{fig-B3}
\end{subfigure}
\caption{Examples of $(T,\sqsubseteq)$.}\label{fig:A3B3}
  \end{figure}
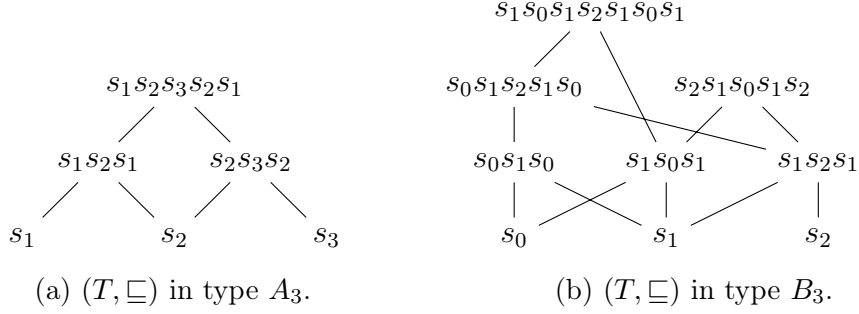

We now prove our first theorem.

\begin{thm} \label{corollario graduato} Let $(W,S)$ be a Coxeter
  system and $X \subseteq T$ be an order ideal of
  $(T,\sqsubseteq)$. Then the poset $(W,\leqslant_{L^X})$ is graded
  with rank function $\ell-\ell\circ Q^{S \cap X}$.
\end{thm}
\begin{proof} 
  First we prove that, if $X \subseteq T$ is an order ideal of
  $(T,\sqsubseteq)$ and $u \var_X v$, then $\ell(v)-\ell(u)=1$, for
  all $u,v\in W$.  Let $u \var_X v$; then $\ell(u)<\ell(v)$ and there
  exists $t \in X$ such that $u=tv$. Assume $\ell(u)<\ell(v)-1$. Then
  there exists a dihedral reflection subgroup $W' \subseteq W$ such
  that $t \in W'$, and a path from $v$ to $u$ in the directed graph
  $\Omega_{W,S}(W'u)=\Omega_{W,S}(W'v)$, whose length is strictly
  greater than $1$. This is Claim $(i)$ in the proof of
  \cite[Proposition 3.3]{Dyer2}. Moreover, if $(W,S)$ is a dihedral
  Coxeter system, $x\in W$ and $t \in T$ such that
  $\ell(tx)>\ell(x)+1$, then there exist $t_1,\ldots,t_n \in T$
  satisfying $x<t_1x < \ldots < t_n\cdots t_1 x= tx$ and
  $\ell(t_i)<\ell(t)$ for all $1 \leqslant i \leqslant n$.  Notice
  that, in a dihedral Coxeter system, $\ell(t')<\ell(t)$ if and only
  if $t' \sqsubset t$, for any $t',t\in T$. Hence there exists a chain
  $u=:x_0 \var x_1 \var \ldots \var x_n:=v$ in $(W,\leqslant)$ such
  that $\ell(x_{i-1}x_i^{-1})<\ell(t)$, for all $i\in [n]$; then
  $x_{i-1}x_i^{-1} \in X$ for all $i\in [n]$, since $X$ is an order
  ideal of $(T,\sqsubseteq)$. Therefore $u \var_X v$ implies
  $\ell(v)-\ell(u)=1$.

 Let $J:=S \cap X$.  If $S \subseteq X$ we have that $J=S$,
 $(\ell \circ Q^J)(w)=0$ for all $w \in W$ and, as posets,
 $(W,\leqslant_L) \hookrightarrow (W,\leqslant_{L^X})$; therefore, in
 this case, the rank function of $(W,\leqslant_{L^X})$ is $\ell$. Let
 us consider the case $S\not\subseteq X$.  First note that $s<t \in X$
 implies $s \in J$, for all $s\in S$.  To see this, first note that
 $t$ has a palindromic reduced expression (see \cite[Lemma
 2.7]{Dyer2}). Since $s<t$, either $t=wsw^{-1}$ with
 $\ell(wsw^{-1})=2l(w)+1$ for some $w\in W$, or $t=wsrsw^{-1}$ for
 some $w\in $ and $r\in T$ with $l(t)=2l(w)+2+l(r)$.  By induction on
 $\ell(w)$, it follows that $s\sqsubseteq t$ in the first case, and
 $srs\sqsubseteq t$ in the second.  We claim $s \sqsubseteq t$ in the
 second case also, For let $W':=\langle s, r \rangle$; therefore
 $s,srs \in W'$ and $s \rightarrow rs \rightarrow srs$ is a path in
 the directed graph $\Omega_{W,S}(W')$, and so
 $s \sqsubseteq srs \sqsubseteq t$.  Therefore $s<t \in X$ implies
 $s \in X$ and then $s\in J$.  So we have that $X \subseteq W_J$ and
 $X$ is an order ideal of $(W_J \cap T, \sqsubseteq)$. Since
 $X \subseteq W_J$, there is a poset isomorphism
 $\left(W_Jw,\leqslant_{L^X}\right) \simeq
 \left(W_J,\leqslant_{L^X}\right)$, for all $w\in {^JW}$. Then we
 conclude using the fact that $W= \biguplus_{w\in {^JW}} W_Jw$.
\end{proof}

\begin{oss}
  In the proof of Theorem \ref{corollario graduato} we see that, if $X \subseteq T$
  is an order ideal of $(T,\sqsubseteq)$, then
  $$(W,\leqslant_{L^X}) \simeq \biguplus_{i=1}^{|^JW|} (W_J,\leqslant_{L^X}),$$ the coproduct of $|^JW|$ copies
  of the poset $(W_J,\leqslant_{L^X})$, where $J:=S \cap X$.
\end{oss}

Clearly, the set $T_k$ is an order ideal of
$\left(T,\leqslant\right)$, and then of $\left(T,\sqsubseteq\right)$;
moreover $S \subseteq T_k$. These observations imply the following
corollary.
\begin{cor}\label{cor:intermediate}
Let $k\in \N$. Then the poset $\left(W,\<\right)$ is graded with rank function $\ell$.
\end{cor}
Note that if $X\supseteq S$ is an order ideal of $T$, then $(W,\X)$ is
graded with rank function $\ell$, and
$(W,\leqslant_L)\hookrightarrow (W,\X)\hookrightarrow (W,\leqslant)$.

A finite graded poset is \emph{strongly Sperner} if no union of $h$
antichains is larger than the union of the $h$ largest rank levels,
for all $h\in \N$.  By a recent result of Gaetz and Gao
\cite{sperner}, the poset $(S_n,\leqslant_L)$ is strongly
Sperner. Since under the hypotheses below any antichain of
$(W,\leqslant_{L^X})$ is an antichain of $(W,\leqslant_L)$, the
following result holds for symmetric groups.

\begin{cor} \label{corollario sperner} Let $X$ be an order ideal of
  $T$ containing $S$.  Then the poset $(S_n,\leqslant_{L^{X}})$ is
  strongly Sperner. In particular, the poset $(S_n,\<)$ is strongly
  Sperner if $n>0$ and $k\in \N$. \end{cor}

The category of posets considered here is the one with posets as
objects and order-preserving functions as morphisms.  Let $(W,S)$ be a
Coxeter system and $J \subseteq S$. Then the function
$Q^J : W \rightarrow W$ is not, in general, an element of
$\End(W,\X)$. For example, in type $A_2$ with Coxeter generators
$\{s,t\}$, we have $ts \leqslant_L sts$ but
$Q^{\{t\}}(ts)=s \nleqslant_L st =Q^{\{t\}}(sts)$.  The second main
result of this article is that, in the same hypotheses as
Theorem~\ref{corollario graduato}, the
functions~${P^J : (W,\X) \rightarrow (W,\X)}$ are order preserving.

\begin{thm} \label{ordine preservato}
  Let $(W,S)$ be a Coxeter system, $J \subseteq S$ and $X \subseteq T$
  be an order ideal of $(T,\sqsubseteq)$. Then $P^J \in \End(W,\X)$.
\end{thm}
\begin{proof}
  We first prove the result in the case $|J|=1$. Let $J=\{s\}$ for
  some $s\in S$, and let $u \var_X v$. Then there exists $t\in X$ such
  that $u=tv$ and, by Theorem \ref{corollario graduato},
  $\ell(v)-\ell(u)=1$. We proceed by induction on $\ell(v)$. If
  $\ell(v)=1$ then $u=e$ and the result is obvious. Let
  $\ell(v)>1$. We consider two cases.
  
 \begin{enumerate}
 \item\label{case1} Suppose $vs<v$. If $u<us$ then $us=v$ and so $P^J(u)=u=P^J(v)$.
                    If $us<u$ then $us=tvs$ and $\ell(tvs)=\ell(vs)-1$; hence
                    $P^J(u)=us \var_X vs=P^J(v)$. 
                  \item Suppose now $v<vs$. If $u_J=e$ then
                    $P^J(u)=u \var_X v = P^J(v)$.  If instead $u_J=s$
                    then $u^J=tvs$ and $\ell(u^J)=\ell(vs)-3$. Hence
                    $u^J \X vs$.  By Theorem \ref{corollario graduato}
                    there exist $w_1,w_2\in W$ such that
                    $u^J \var_X w_1 \var_X w_2 \var_X vs$.  If
                    $w_2<w_2s$ then $w_2=v=v^J$ and the result
                    follows. So let $w_2s<w_2$. Since
                    $\ell(w_2)=\ell(v)$, by case~\ref{case1} above and
                    our inductive hypothesis
                    $u^J \X P^J(w_1) \X P^J(w_2)$. Moreover,
                    $P^J(w_2)=w_2s=(rvs)s=rv$, for some $r\in X$, and
                    $\ell(w_2s)=\ell(v)-1$. Hence
                    $u^J \X P^J(w_1) \X P^J(w_2) \var_X v$.
\end{enumerate}

Let $u\X v$ such that $\ell(u)<\ell(v)-1$. Then, by Theorem
\ref{corollario graduato} there exists a chain
$u \var_X w_1 \var_X \ldots \var_X w_n \var_X v$. Therefore
  $$P^J(u) \X P^J(w_1) \X \ldots \X P^J(w_n) \X P^J(v).$$

  Let now $|J|>1$ and $u \var_X v$. We proceed by induction on
  $\ell(v)$. If $\ell(v)=1$ then $u=e$ and the result is obvious. Let
  $\ell(v)>1$ and consider the following two cases.
  \begin{enumerate}
  \item Suppose $v_J>e$. Let $s\in D_R(v_J)$. If $u<us$ we have that
    $us=v$ and then $P^J(u)=P^J(us)=P^J(v)$.  If $us<u$, as before,
    $us \var_X vs$. By our inductive hypothesis
    $P^J(u)=P^J(us) \X P^J(vs)=P^J(v)$.
                  \item Suppose now $v_J=e$. If $u_J=e$, then
                    $P^J(u)=u \var_X v = P^J(v)$. If $u_J>e$, let
                    $s_1\cdots s_m$ be a reduced word for $u_J$. Then
                    $\{s_1,s_2,\ldots,s_m\} \subseteq J$ and the
                    result follows from the case $|J|=1$, since
$$u^J=P^{\{s_1\}}\cdots P^{\{s_m\}}(u^Ju_J) \X P^{\{s_1\}}\cdots P^{\{s_m\}}(v)=v. \qedhere$$\end{enumerate}
\end{proof}

By the same arguments as before, the previous result holds in
particular for the orders $(W,\<)$, thus proving the second theorem in
Section \ref{intro}.
\begin{oss}
  The previous result is known for the Bruhat order (see
  \cite[Proposition 2.5.1]{BB}); in that case it is a direct
  consequence of the lifting property. For $X=S$ the result, in its
  right version, is \cite[Lemma 2.1]{Denfant}.
\end{oss}

Recall that given a Coxeter system $(W,S)$ with Coxeter matrix $m$,
one can define an associated Coxeter monoid $W^m$ as
$$W^m=\left\langle s_i\in S : s_i^2=s_i, \underbrace{s_i s_j s_i \dots}_{m(s_i,s_j) \text{ terms}} =  \underbrace{s_j s_i s_j \dots }_{m(s_i,s_j) \text{ terms}} \right\rangle,$$ see also \cite{Kenney}. It is known that the Coxeter monoid $W^m$ is a submonoid
of the monoid $M$ generated by the functions $\{P^J:J \subseteq S\}$;
see e.g.~\cite[Section 4]{Sentinelli-Artin} for details.  By Theorem
\ref{ordine preservato}, $M$ is a submonoid of $\End(W,\<)$. Hence we
have the following corollary.
\begin{cor}
Let $k\in \N$. Then, as monoids,  $W^m \hookrightarrow \End(W,\<)$.
\end{cor}

Let $k\in \N$, $(W,\{s_1,\ldots,s_n\})$ be a Coxeter system and
let
$$\phi_k : (W,\leqslant_{L^k}) \rightarrow
(W^{S\setminus\{s_1\}},\leqslant_{L^k}) \times \ldots \times
(W^{S\setminus\{s_n\}},\leqslant_{L^k})$$ be the function defined by
$\phi_k(w)=(P^{S\setminus\{s_1\}}(w),\ldots,P^{S\setminus\{s_n\}}(w))$,
for all $w\in W$.  By Theorem \ref{ordine preservato} the function
$\phi_k$ is order preserving. Given a function
$f : (A,\leqslant_A) \rightarrow (B,\leqslant_B)$, we define an
induced subposet of $B$ by
$\mathrm{Im}(f):=\left( \{f(a):a\in A\}, \leqslant_B \right)$.

\begin{prop}
  Let $W=S_n$ with its standard Coxeter presentation. Then
  $\mathrm{Im}(\phi_k)\simeq \left(S_n,\leqslant\right)$, for all
  $k\in \N$.
\end{prop}
\begin{proof}
 It suffices to prove the result for $k=0$. By \cite[Exercise 3.2]{BB}, 
 $$(S_n^{(1)},\leqslant_L) \times \ldots \times (S_n^{(n-1)},\leqslant_L)=(S_n^{(1)},\leqslant) \times \ldots \times (S_n^{(n-1)},\leqslant),$$ where, for $h\in [n-1]$, we have defined $S^{(h)}_n:=S_n^{S\setminus \{s_h\}}$.
 Since $u \leqslant v$ if and only if
 $P^{S\setminus \{s_h\}}(u) \leqslant P^{S\setminus \{s_h\}}(v)$ for
 all $h\in [n-1]$ (see, e.g. \cite[Theorem 2.6.1]{BB}), the result
 follows.
\end{proof}

The previous result could be not true for other Coxeter groups. For
example, in type $B_3$, the poset $\mathrm{Im}(\phi_0)$ is not graded.

\section{The $k$-absolute orders and open problems}\label{sec:pbs}

In this section we define $k$-analogues of Bruhat graphs, absolute
orders and absolute length. We then formulate a few related
conjectures and open problems. These definitions could be extended to
orders $\leqslant_{L^{X}}$ where $X$ is an order ideal of $T$
containing $S$, but the conjectures are not stated in that extended
setting since they are inadequately tested in that generality.

Let $(W,S)$ be a Coxeter system, $k\geqslant 0$ and $\Omega^k$ the
directed graph whose vertex set is $W$ and such that there is an arrow
from $a$ to $b$ if and only if $ab^{-1}\in T_k$ and $a<b$. We call
$\Omega^k$ the $k$-\emph{Bruhat graph} of $(W,S)$. Let
$\vec{d}_k(a,b)$ be the distance from $a$ to $b$ in the directed graph
$\Omega^k$. We define the $k$-\emph{absolute length} of $w\in W$ by
$$\ell_k(w):=\vec{d}_k(e,w).$$ Clearly $\ell_0=\ell$
and, in the finite case for $k$ big enough, $\ell_k=a\ell$, where
$a\ell$ is the absolute length (see, for instance, \cite[Exercise
7.2]{BB} and references therein).

\begin{dfn} \label{def k-assoluto} Let $k\in \N$.  The left
  $k$-\emph{absolute order} $\preccurlyeq_k$ on $W$ is the partial
  order defined by letting $u \preccurlyeq_k v$ if and only if
  $\ell_k(v)=\ell_k(u)+\ell_k(vu^{-1})$, for all $u,v\in W$.
\end{dfn} By definition, the poset $(W,\preccurlyeq_k)$ is ranked with
rank function $\ell_k$. For $k=0$ we recover the left weak order; in
the finite case, for $k$ big enough we obtain the \emph{absolute
  order} $\preccurlyeq$ first introduced in \cite{brady}.
\begin{oss}
  Notice that the maximal chains of $(W,\preccurlyeq_k)$ could have
  different lengths. For example, if $k=1$ and $W=S_4$ with its
  standard Coxeter presentation, then
  $\max_{\preccurlyeq_k}W=\{2413, 3142, 4321\}$,
  $\ell_1(2413)=\ell_1(s_1s_3s_2)=3$,
  $\ell_1(3142)=\ell_1(s_2s_3s_1)=3$ and
  $\ell_1(4321)=\ell_1(s_1s_2s_3t)=4$, where $t:=s_1s_2s_1$.
\end{oss}

We denote by $[u,v]_a$ an interval in the absolute order and by
$[u,v]_{ak}$ an interval in the $k$-absolute order. Brady proved in
\cite{brady} that, in type $A_n$, if $c$ is any Coxeter element then
the interval $[e,c]_a$ is isomorphic to the lattice of noncrossing
partitions. In any finite type, the intervals $[e,c]_a$ have been
proved to be shellable in \cite{Athanasiadis}. In this vein, we
formulate the following conjecture for $k$-intermediate orders and
$k$-absolute orders.
\begin{conge} \label{congettura shellable} Let $c$ be a Coxeter
  element of a Coxeter system $(W,S)$. Then the order complexes of the
  intervals $[e,c]_k$ and $[e,c]_{ak}$ are shellable, for all
  $k\in \N$.
\end{conge} Since Bruhat intervals are shellable, the conjecture is
true for Bruhat intervals $[e,c]$.  By SageMath computations, we have
verified the conjecture for the intervals $[e,c]_k$ in type $A_n$, for
all $1 \leqslant n\leqslant 5$, in type $B_n$, for all
$2 \leqslant n\leqslant 5$, in types $D_4$, $D_5$, $F_4$, $H_3$ and
$H_4$, for all $k\in \N$.  We have verified the conjecture for the
intervals $[e,c]_{ak}$ in type $A_n$, for all
$1 \leqslant n\leqslant 5$, in type $B_n$, for all
$2 \leqslant n\leqslant 4$, in types $D_4$ and $H_3$, for all
$k\in \N$.

Another feature of an absolute order is its strong Sperner property.
It has been proved in \cite{harper}, using flow techniques on Hasse
diagrams, that the poset $(S_n,\preccurlyeq)$ is strongly
Sperner. This result has been stated in \cite{sperner2} for finite
Coxeter groups, except for type $D_n$. Hence we can propose the
following problem for finite Coxeter groups.

\begin{prob}
Study the Sperner properties of $(W,\preccurlyeq_k)$.
\end{prob}
In the previous section we mentioned that $(S_n,\preccurlyeq_0)$ is
strongly Sperner and from this fact we could deduce Corollary
\ref{corollario sperner}. For other finite Coxeter groups and $k=0$
the strong Sperner property is expected in
\cite[Conjecture~3.1]{sperner}.

\vspace{.5em}

We now consider the distribution of $\ell_k$ on any Coxeter group. 

\begin{ex}
  The generating function of $\ell_1$ on $S_4$ is
  $$\sum_{w\in S_4}x^{\ell_1(w)}=1+5x+10x^2+7x^3+x^4.$$
\end{ex}

\begin{oss}
  The coefficient of $x$ in $\sum_{w\in W}x^{\ell_k(w)}$ is the number
  of reflections in $T_k$. For $W=S_n$, this turns out to be
  $|T_k|=\binom{n}{2}-\binom{n-k-1}{2}$. Indeed, it is easy to see
  that the number of transpositions in $T\setminus T_k$, that is the
  number of transpositions with length strictly greater than $2k+1$,
  is $\binom{n-k-1}{2}$.
\end{oss}

For definitions and results about log-concavity and unimodality we
refer to \cite{Stanley-log} and \cite{Brenti-log}.  We put forward two
conjectures about the generating functions of the $k$-absolute length
on finite Coxeter systems. The first is as follows.
\begin{conge} \label{congettura 1} Let $(W,S)$ be a finite Coxeter
  system and $k\geqslant 0$. Then the polynomial
  $\sum_{w\in W}x^{\ell_k(w)}$ is log-concave with no internal zeros.
\end{conge}

By \cite[Proposition 2]{Stanley-log} and the known factorization of the 
polynomials $\sum_{w\in W}x^{\ell(w)}$ and $\sum_{w\in W}x^{a\ell(w)}$
(see, e.g. \cite[Theorem 7.1.5]{BB} and \cite[Exercise 7.2]{BB}),
the previous conjecture holds in these two extreme cases. 

For dihedral groups $I_2(m)$, the statement of Conjecture
\ref{congettura 1} holds. Indeed, the following formula can be
directly verified.  Let $h\in \N\setminus \{0\}$ and define the
function $\pi_h: \N \rightarrow \N$ by
$\pi_h(n)=n-h\lfloor n/h \rfloor$, for all $n \in \N$.  If
$1 \leqslant 2k+1\leqslant m$, then
 $$\sum\limits_{w\in I_2(m)}x^{\ell_k(w)} =1+2(k+1)x+2(2k+1)\sum\limits_{i=2}^{\left\lfloor\frac{m-1}{2k+1} \right\rfloor}x^i+a_{k,m}x^{\left\lfloor\frac{m-1}{2k+1} \right\rfloor+1} +b_{k,m}x^{\left\lfloor\frac{m-1}{2k+1} \right\rfloor+2},$$
 where
 $$a_{k,m}:= 2k+\pi_{2k+1}(m)+\left\{
                 \begin{array}{ll}
                   2k+1, & \hbox{if $\pi_{2k+1}(m)=0$;} \\
                   0, & \hbox{otherwise,}
                 \end{array}
               \right.$$ and
  $$b_{k,m}:= \left\{
                 \begin{array}{ll}
                   2k, & \hbox{if $\pi_{2k+1}(m)=0$;} \\
                   \pi_{2k+1}(m)-1, & \hbox{otherwise.}
                 \end{array}
               \right.$$

Explicit computations carried out with SageMath \cite{SageMath}
confirm Conjecture~\ref{congettura 1} for all relevant~$k$ for types
$A_n$ ($n\leq 5$), $B_n$ ($n\leq 4$), $D_n$ ($n\leq 5$), $F_4$ and
$H_3$.  Since a non-negative log-concave sequence with no internal
zeros is unimodal, the previous conjecture implies the following.

\begin{conge}
  Let $(W,S)$ be a finite Coxeter system and $k\geqslant 0$. Then the
  polynomial $\sum_{w\in W}x^{\ell_k(w)}$ is unimodal.
\end{conge}

We end this section with a problem on the Ricci curvature of
$k$-Bruhat graphs. The Ricci curvature of the graph $\Omega^0$ (the
Cayley graph of $(W,S)$) is studied (and in many cases explicitly
computed) in \cite{Siconolfi-1}; the Ricci curvature of the Bruhat
graph of a finite Coxeter group is proved to be $2$ in
\cite{Siconolfi-2}. We refer to these articles for definitions and
preliminary results.  Note that a $k$-Bruhat graph is always
locally-finite and is a subgraph of the Bruhat graph; it is then
natural to formulate the following problem.

\begin{prob}
Let $(W,S)$ be a Coxeter system and $k\geqslant 0$. 
What can be said about the Ricci curvature of the $k$-Bruhat graph?
\end{prob}

\section*{Acknowledgements}
 The third author thanks the School of Mathematical and Statistical
 Sciences of the National University of Ireland in Galway, for its
 hospitality during summer 2019.  The first author was partially supported by an
 Irish Research Council postdoctoral fellowship (grant no.\
 GOIPD/2018/319).

\bibliographystyle{abbrv}
\bibliography{kBruhat.bib}

\begin{thebibliography}{10}

\bibitem{Athanasiadis}
C.~A. Athanasiadis, T.~Brady, and C.~Watt.
\newblock Shellability of noncrossing partition lattices.
\newblock {\em Proc. Amer. Math. Soc.}, 135(4):939--949, 2007.

\bibitem{Athanasiadis2}
C.~A. Athanasiadis and M.~Kallipoliti.
\newblock The absolute order on the symmetric group, constructible partially
  ordered sets and {C}ohen-{M}acaulay complexes.
\newblock {\em J. Combin. Theory Ser. A}, 115(7):1286--1295, 2008.

\bibitem{BB}
A.~Bj{\"o}rner and F.~Brenti.
\newblock {\em Combinatorics of {C}oxeter groups}, volume 231 of {\em Graduate
  Texts in Mathematics}.
\newblock Springer, New York, 2005.

\bibitem{brady}
T.~Brady.
\newblock A partial order on the symmetric group and new {$K(\pi,1)$}'s for the
  braid groups.
\newblock {\em Adv. Math.}, 161(1):20--40, 2001.

\bibitem{Brenti-log}
F.~Brenti.
\newblock Log-concave and unimodal sequences in algebra, combinatorics, and
  geometry: an update.
\newblock In {\em Jerusalem combinatorics '93}, volume 178 of {\em Contemp.
  Math.}, pages 71--89. Amer. Math. Soc., Providence, RI, 1994.

\bibitem{Denfant}
C.~Defant.
\newblock {Pop-Stack-Sorting for Coxeter Groups}.
\newblock arXiv:2104.02675, 2021.

\bibitem{Dyer}
M.~Dyer.
\newblock Reflection subgroups of {C}oxeter systems.
\newblock {\em Journal of Algebra}, 135(1):57--73, 1990.

\bibitem{Dyer2}
M.~Dyer.
\newblock On the ``{B}ruhat graph'' of a {C}oxeter system.
\newblock {\em Compositio Mathematica}, 78(2):185--191, 1991.

\bibitem{sperner}
C.~Gaetz and Y.~Gao.
\newblock A combinatorial {$\mathfrak{sl}_2$}-action and the {S}perner property
  for the weak order.
\newblock {\em Proc. Amer. Math. Soc.}, 148(1):1--7, 2020.

\bibitem{sperner2}
C.~Gaetz and Y.~Gao.
\newblock On the {S}perner property for the absolute order on complex
  reflection groups.
\newblock {\em Algebr. Comb.}, 3(3):791--800, 2020.

\bibitem{harper}
L.~H. Harper and G.~B. Kim.
\newblock The symmetric group, ordered by refinement of cycles, is strongly
  {S}perner.
\newblock {\em Proc. Amer. Math. Soc.}, 149(7):2753--2761, 2021.

\bibitem{Kenney}
T.~Kenney.
\newblock Coxeter groups, {C}oxeter monoids and the {B}ruhat order.
\newblock {\em J. Algebraic Combin.}, 39(3):719--731, 2014.

\bibitem{spectra}
P.~Renteln.
\newblock The distance spectra of {C}ayley graphs of {C}oxeter groups.
\newblock {\em Discrete Math.}, 311(8-9):738--755, 2011.

\bibitem{Sentinelli-Artin}
P.~Sentinelli.
\newblock Artin group injection in the {H}ecke algebra for right-angled groups.
\newblock {\em Geom. Dedicata}, 214:193--210, 2021.

\bibitem{Siconolfi-1}
V.~Siconolfi.
\newblock Ricci curvature, graphs and eigenvalues.
\newblock {\em Linear Algebra Appl.}, 620:242--267, 2021.

\bibitem{Siconolfi-2}
V.~Siconolfi.
\newblock {Ricci curvature, Bruhat graphs and Coxeter groups}.
\newblock arXiv:2102.11277, to appear in {\em Proc. Amer. Math. Soc.}, 2022.

\bibitem{Stanley-log}
R.~P. Stanley.
\newblock Log-concave and unimodal sequences in algebra, combinatorics, and
  geometry.
\newblock In {\em Graph theory and its applications: {E}ast and {W}est
  ({J}inan, 1986)}, volume 576 of {\em Ann. New York Acad. Sci.}, pages
  500--535. New York Acad. Sci., New York, 1989.

\bibitem{StaEC1}
R.~P. Stanley.
\newblock {\em Enumerative combinatorics. {V}ol. 1}, volume~49 of {\em
  Cambridge Studies in Advanced Mathematics}.
\newblock Cambridge University Press, Cambridge, 2012.
\newblock Second edition.

\bibitem{SageMath}
{The Sage Developers}.
\newblock {\em {S}ageMath, the {S}age {M}athematics {S}oftware {S}ystem
  ({V}ersion 9.4)}.

\end{thebibliography}
\end{document}